\documentclass[11pt]{amsart}
\usepackage{a4wide}
\usepackage{amssymb,amsfonts}

\usepackage{amsmath,amsthm}

\usepackage[initials,short-publishers,lite]{amsrefs}
\usepackage{verbatim}
\usepackage{bm}
\usepackage{bbm}
\usepackage{setspace}
\usepackage{textcomp}
\usepackage[dvipsnames]{xcolor}

\renewcommand{\b}{b}
\newcommand{\Ao}{\mathcal{A}}
\newcommand{\At}{\mathcal{B}}
\newcommand{\Zp}{{\mathbb{Z}_+}}
\newcommand{\Oo}{\Omega}
\newcommand{\Ot}{\Sigma}
\newcommand{\Po}{\mathbb{\mu}}
\newcommand{\Pt}{\mathbb{\nu}}

\newcommand{\N}{\mathbb{N}}
\newcommand{\Z}{{\mathbb Z}}

\newcommand{\R}{{\mathbb R}}
\renewcommand{\P}{\mathbb{P}}
\newcommand{\gfunc}{\mathcal{G}}
\newcommand{\gibbs}{\mathcal{G}}
\newcommand{\cyltjur}{\mathfrak{M}}
\newcommand{\net}{\mathcal{N}}
\newcommand{\tjur}{\overline{\mathfrak{M}}}

\newcommand{\var}{\mathrm{var}}

\newcounter{para}

\def\mybf #1{{\bf{\emph{#1}}}}
\numberwithin{equation}{section}
\newtheorem{theorem}{Theorem}[section]
\newtheorem{proposition}[theorem]{Proposition}
\newtheorem{lemma}[theorem]{Lemma}

\newtheorem{corollary}[theorem]{Corollary}
\newtheorem{definition}[theorem]{Definition}

\theoremstyle{remark}
\newtheorem{remark}{Remark}[section]

\title{On regularity of functions of Markov chains}
\author{Steven Berghout}
\address{S. Berghout: Mathematical Institute, University of Leiden, PO Box 9512, 2300 RA Leiden, The Netherlands}
\author{ Evgeny Verbitskiy}
\address{ E.Verbitskiy: Mathematical Institute, University of Leiden, PO Box 9512, 2300 RA Leiden, The Netherlands \newline\indent \textup{and} \newline\indent Bernoulli Institute, University of Groningen, PO Box 407, 9700 AK, Groningen, The Netherlands}

\begin{document}
\keywords{Functions of Markov measures\and Hidden Markov processes \and $g$-measures }
\subjclass{60J10 \and 37A50\and 37D35 }
\maketitle
\begin{abstract}
 We consider processes which are functions of finite-state Markov chains. It is well known that such processes are rarely Markov. However, such processes are often \emph{regular} in the following sense: 
the distant past values of the process have diminishing influence on the distribution of the present value.
In the present paper, we present  novel sufficient conditions for regularity of functions of Markov
chains. 
\end{abstract}
\section{Introduction}
Suppose $\{X_n\} $ is a stationary Markov chain taking values in a finite set $\Ao$ and assume that we are not able to observe the values $\{X_n\}$ directly. Instead, we observe the values of some function of $X_n$ that groups some elements in $\Ao$ together. To be precise, let $\pi:\Ao\to \At$, with $\At$ a smaller alphabet, and assume that we observe the process $\{Y_n\}$ given by
\begin{equation}
\label{eqt:factorprocess}
 Y_n=\pi(X_n),\quad\text{for all }n.
\end{equation}

Processes of this form have been studied extensively in the past 60 years and appear under a variety
of different names in various fields: in Probability Theory, functions of Markov chains \cite{Burke58}, grouped \cite{Harris55}, lumped \cite{Kemeny60}, amalgamated \cite{Chazottes03}, or aggregated Markov chains \cite{Stadje05}; one-block factors of Markov measures \cite{Marcus84} or sofic measures \cite{Kitchens85} in Ergodic Theory, Fuzzy Markov measures \cite{Maes95} in Statistical Mechanics. Note also that the Hidden Markov models \cite{Baum66} -- very popular in Statistics, can be cast in the form (\ref{eqt:factorprocess}) as well.

The factor process $\{Y_n\}$ is rarely Markov, the necessary and sufficient conditions have been found by Kemeny \& Snell and Dynkin \cites{Kemeny60,Dynkin63}.
This raises the principal question: what is the dependence structure of the factor process?

It turns out that, under rather mild conditions on the underlying Markov chain and the coding map $\pi$,
the resulting process can be seen as \emph{approximately} or \emph{nearly} Markov in the following sense:
the conditional distribution of the next value $Y_1$ depends on the complete past $Y_{-\infty}^0:=(\ldots, Y_{-2}, Y_{-1},Y_{0})$, but this dependence is \emph{regular}, i.e.,  the distant past values $\{Y_{-n}\}$, for $n\gg 1$,
have a diminishing effect on the distribution of $Y_1$. 
Stochastic processes with such properties occur naturally in many contexts, as a consequence many authors introduced concepts that formalize the notion of a measure that is approximately Markov. Among these concepts are the chains with complete connections \cite{Onicescu35}, chains of infinite order \cite{Harris55}, $g$-measures \cite{Keane72} and uniform martingales \cite{Kalikow90}. Although these concepts are very similar, they are not always equivalent, for a more detailed discussion see \cite{Fernandez04}. Among these notions $g$-measures are the most convenient for the purposes of this paper. Usually $g$-measures are defined on some subset of the product space $\Ao^\Zp$. This space can be thought of as the collection of all allowed paths of a process starting at time $0$. Like a Markov measure, a $g$-measure is introduced via its transition probabilities, the differences are that the transitions of a $g$-measure are described by a function $g:\Ao^ \Zp \to (0,1)$, rather than a matrix $P:\Ao \times \Ao \to [0,1]$ and that the time direction is reversed. That is, the vector $g(\cdot \  x_1^\infty)$ represents the distribution of the symbol in the origin, conditioned on the `future' configuration $x_1^\infty$. This time reversal is common in ergodic theory and is mostly inconsequential for our purposes, as a Markov measure satisfies the Markov property in both directions. A $g$-measure is approximately Markov due to the additional constraint that the function $g$ is continuous. To clarify, continuity corresponds to a vanishing influence from far away symbols since, in the product topology, a function $g:\Ao^\Zp \to \R$ is continuous if and only if:
\[
\var_n(g) \equiv \sup_{x,y \in \Ao^\Zp} \left| g(x_{0}^\infty) - g(x_0^n y_{n+1}^\infty)\right| \to 0 \mathrm{, \ as\ }n \to \infty.
\]

We can now use the language of $g$-measures to phrase the main result of this paper: we provide a novel sufficient condition for functions (factors) of Markov chains to belong to the class of $g$-measures. This condition is based on the application of the so-called fibre approach that originated in Ergodic Theory \cite{Ledrappier77}, but seems to be less known in Probability Theory.

Let us now describe this method briefly. Suppose $\{X_n\}$  is a $\Ao$-valued stationary process, and $\mu$ is its translation invariant measure on $\Ao^\Zp$.
Denote by   $\nu=\mu\circ\pi^{-1}$ the stationary law of the $\At$-valued factor process $Y_n=\pi(X_n)$.
Define the \emph{fibre} over $y \in \Ot$ as the set $\Oo_y = \{ x \in \Ao^\Zp: \pi(x_n) =y_n \text{ for all } n\ge 0\}$. 
By a well-known Disintegration Theorem, there exists a family of measures $\{ \Po_y \}$ indexed by points $y\in\At^\Zp$, called a \emph{disintergation} of $\mu$, such that $\Po_y$ is concentrated on the fibre 
$\Oo_y$ and $\Po = \int \Po_y d\Pt$, meaning that for any $\Po$-integrable function $f$ on $\Ao^{\Zp}$, one has 
$$
\int f(x) \mu(dx) =\int \Bigl[ \int f(x) \mu_y(dx)\Bigr] \nu(dy).
$$
The measures $\mu_y$ should be viewed as \emph{conditional measures} on fibres $\Oo_y$.

We will show in Theorem \ref{thm:decomp} that the factor measure $\Pt=\Po\circ\pi^{-1}$ of a Markov measure $\Po$ is consistent with a function $\tilde{g}$, i.e., $\Pt(y_0|y_1^\infty) = \tilde{g}(y)$ for $\Pt$-a.a. $y$, where $\tilde{g}$ can be expressed in terms of the disintegration $\{ \Po_y \}$. 
Based on the expression for $\tilde g$, we identify two sufficient conditions for $\Pt$ to be a $g$-measure. The first condition turns out to be equivalent to a known condition of weak lumpability, i.e., covers the case
 when the factor measure  $\Pt$ is Markov. 
The second sufficient condition is that the disintegration can be chosen in such a way that $y \to \Po_y$ is continuous. 
In this case, we say that $\Po$ admits a  Continuous Measure Disintegration (CMD) $\{ \Po_y \}$. Our main result is that existence of a CMD for $\Po$ implies that the factor measure is a a $g$-measure.

Previously, this condition has been applied successfully to the analogous question; when is a factor of a fully supported $g$-measure itself a $g$-measure \cites{Verbitskiy11,Johansson17}. In the context of factors of Markov measures, we show that the condition supersedes the currently known conditions.

These results are presented here in the following way: firstly, we introduce the necessary definitions, then we review known results in sections \ref{sec:markovtomarkov}, \ref{sec:fullypositive} and \ref{sec:fibremixing}. Subsequently we state our main  theorem in Section \ref{sec:mainresults}. In order to demonstrate that the condition in Section \ref{sec:mainresults} is more general than known results we apply the theory of non-homogeneous equilibrium states in Section \ref{sec:nonhomogeneouseqstates}. We will also recall the constructive approach to continuous measure disintegrations by Tjur in section \ref{sec:tjurformalism} to provide an interesting alternative to recover the known conditions in \ref{sec:conditionalmeasuresaremarkov}.
Finally, in Section \ref{sec:examples}, we discuss some examples to show that existence of a continuous measure disintegration is \emph{strictly} weaker than the previously known conditions and that, unfortunately, it is not a necessary condition.

\subsection{Notation}
\label{sec:notation}
Suppose $\Ao$ is a finite set (alphabet) and $M$ is  $|\Ao|\times|\Ao|$ matrix with entries in $\{0,1\}$.
The corresponding subshift of finite type (SFT) $\Oo_M$ is defined as
$$
 \Oo_M=\left\{ x=(x_n)_{n=0}^\infty \in \Ao^\Zp : M({x_n,x_{n+1}})=1 \quad \forall n \in \Zp\right\}.
$$
We equip $\Oo_M$ with the product topology. We use the shorthand notation $a_n^m=(a_n,a_{n+1},\ldots,a_m)$ for words in alphabet $\Ao$, and denote the corresponding
cylinder sets as $ [a_n^m] = \left\{ x \in \Oo_M : x_n^m = a_n^m \right\}$.
Similarly, for a given finite set $\Lambda \subset \Zp$, denote the configuration on the subset $\Lambda$ by $a_\Lambda = (a_i)_{i \in \Lambda}$. A concatenation of two configurations $a_\Lambda$ and $b_\Delta$ on disjoint sets $\Lambda,\Delta \subset \Zp$ is denoted as $a_\Lambda b_\Delta$, to be precise:
\[
(a_\Lambda b_\Delta)_i = \left\{ 
\begin{array}{ll} 
a_i & \mathrm{,if\ } i \in \Lambda, \\ 
b_i & \mathrm{,if\ } i \in \Delta.
\end{array} \right. 
\]
For a given subshift of finite type $\Oo_M$ a Markov chain with probability transition matrix $P$
is said to be \emph{compatible} with $\Oo_M$ if $P_{ij}>0 \iff M_{ij}=1$ for all $i,j\in\Ao$. In complete analogy with the terminology for Markov chains, the  subshift of finite type $\Oo_M$ is called
\begin{itemize}
\item[(a)] \mybf{irreducible} if  $\forall i,j\in\Ao$ there exists an $n=n(i,j)>0$ such that
$M^n(i,j)>0$;
\item[(b)] \mybf{aperiodic} if $\forall i\in \Ao$, one has $\gcd \left\{m>0: M^m(i,i) > 0 \right\} = 1$;
\item[(c)] \mybf{primitive} if there exists an $n>0$ such that $M^n>0$.
\end{itemize}
If the subshift of finite type $\Oo_M$ is irreducible, and $P$ is a compatible probability transition matrix,
then the (unique) stationary Markov measure $\mu$ has $\Oo_M$ as its support.

 \subsection{Single block factor maps}
Suppose $\Ao$ and $\At$ are finite sets, $|\Ao|>|\At|$, and $\pi:\Ao \to \At$ is a surjective map.
We use the same symbol $\pi$ to denote the map from $\Ao^{\Zp}$ to $\At^{\Zp}$ given by
$\pi(x)_n= \pi(x_n)$ for all $n\in\Zp$.
Let $\Po$ be a stationary Markov measure corresponding to a Markov chain $\{X_n\}$, supported on an irreducible subshift of finite type $\Oo=\Oo_M\subset \Ao^{\Zp}$, define the push-forward (or factor) measure $\Pt$ as $\Pt = \Po \circ \pi^{-1}$. The measure $\nu$ is supported on a subshift $\Ot = \pi(\Oo) \subset \At^\Zp$.
In symbolic dynamics, $\Ot$ and $\nu$ are called the sofic shift and the sofic measure, respectively.
Note that $\Ot$ is not necessarily a subshift of finite type. Throughout the paper
we make the following standing assumptions on $\Oo$ and $\pi$:
\begin{equation}\label{assumeSFTone}\tag{A1}
\text{ $\Oo=\Oo_M$  is an irreducible SFT,}
 \end{equation}
\begin{equation}\label{assumeSFTtwo}\tag{A2}
\text{the one block factor map $\pi:\Ao^{\Zp}\to\At^{\Zp}$ is such that
 $\Ot=\pi(\Oo)$ is an SFT}
 \end{equation}
 i.e.,  $\Ot=\Ot_{M'}$ for some $\{0,1\}$ matrix $M'$. We note that using standard methods of symbolic dynamics (Fisher covers), it is possible to
decide algorithmically whether for a given pair $(\Oo, \pi)$, the image $\Ot$ is indeed an SFT \cite{Lind95}.

{
 \begin{remark}\label{remhigher} The class of 1-block factor maps $\pi$ between subshifts of type, which we consider in the present paper, is in fact quite large. Indeed, by the Curtis-Lyndon-Hedlund theorem any equivariant map $\pi:\Oo\to\Ot$ is necessarily a sliding block code of some finite length, i.e.,
 there exists $k\ge 0$ such that if $y=\pi(x)$, then for any $n\ge 0$,
 $y_n$ is a function of $x_{n}^{n+k}$. By going to a higher block representation
 (c.f., \cite[Section 1.4]{ Lind95} ) $\Oo^{[k]}$ of $\Omega$, one immediately concludes that 
 the $k$-block factor map $\pi:\Oo\to\Ot$ can be equivalently represented
 as a $1$-block factor map $\pi^{[k]}:\Oo^{[k]}\to\Ot$. Finally, note also that if $\Pt$ is a Markov measure on the SFT $\Omega$, then the corresponding measure $\mu^{[k]}$ on $\Oo^{[k]}$ is again Markov, and hence, $\Pt=\mu^{[k]}\circ (\pi^{[k]})^{-1}$ is a 1-block factor 
 of the Markov measure $\mu^{[k]}$.

 \end{remark}
}
\subsection{$g$-measures}
\label{sec:gmeasures}

As we will see below, factors of Markov measures are rarely Markov.
Instead, it is far more common for factors of Markov measures to belong to the class of $g$-measures, i.e., measures having  positive continuous conditional probabilities. Suppose $\Ot\subseteq\At^{\Z_+}$ is a SFT and consider the following set of functions:
\[
\gfunc(\Ot)=\left\{ g \in C(\Ot,(0,1)): \sum_{b\in \At: by\in \Ot} g(by) = 1  \text{ for all }y \in \Ot \right\}.
\]
\begin{definition}
A translation-invariant measure $\Pt$ on $\Ot$ is called a $g$-measure for $g \in \gfunc(\Ot)$ if
\[
\Pt(y_0|y_1^\infty) = g(y_0^\infty)
\]
for $\Pt$-a.e. $x \in \Ot$. Equivalently, $\Pt$ is a $g$-measure if, for any continuous function $f:\Ot \to \R$, one has
$$
\int f(y) \Pt(dy) =\int \sum_{b \in \At: \  by\in\Ot} f(by) g(by) \Pt(dy).
$$
\end{definition}
For any $g \in \gfunc(\Ot)$, at least one $g$-measure exists; however such a measure might not be unique \cite{Bramson93}. A useful property of $g$-measures is that they are characterized by the uniform convergence of finite one-sided conditional probabilities.
\begin{proposition}\cite{Palmer77} A translation invariant probability measure $\Pt$ on a SFT $\Ot$ is a $g$-measure
if and only if the sequence of local functions on $\Ot$
$$
 g_n(y_0^n) := \nu(y_0| y_1^n), \quad
$$
converges uniformly to some function $g\in \gfunc(\Ot)$.
\end{proposition}
In the opposite direction, one can conclude that a given measure $\Pt$ is not a $g$-measure
if one is able to find a  so-called \emph{bad configuration} for $\Pt$.
\begin{definition}
A point $y\in\Ot$ is called a bad configuration for $\Pt$ if there exists an $\epsilon>0$ such that, for every $n\in\mathbb N$, one can find two points $\underline y,\overline y\in \Ot$ and $m\in\mathbb N$
such that
$$
y_0^n = \underline y_0^n =\overline y_0^n
$$
and
$$
\nu(y_0| y_1^n \overline y_{n+1}^{n+m}) - \nu(y_0| y_1^n \underline y_{n+1}^{n+m}) \ge \epsilon>0.
$$
\end{definition}
Existence of a bad configuration $y$ implies that no version of the conditional probabilities $\nu(y_0|y_1^\infty)$ (defined $\nu$-a.s.), can be continuous at $y$, and hence $\Pt$ cannot be a $g$-measure
for any continuous $g\in\gfunc(\Ot)$.

\section{Properties of factors of Markov measures}
Despite drawing significant interest in various fields, the problem of finding the necessary and sufficient conditions for factors of Markov measures to be regular is still open. On the other hand, the question under which conditions factors of Markov measures 
are Markov have been answered completely in 1960's.

\subsection{Markov factors of Markov measures}
\label{sec:markovtomarkov}
Note that the Markovianity of the factor measure might depend on the initial distribution of the underlying
Markov chain. The notion of \emph{lumpability} was developed to address this question in a uniform fashion, i.e., independently of the initial distribution.

Let $P$ be a stochastic matrix, indexed by $\Ao \times \Ao$ and $\pi:\Ao \to \At$ a factor map. Now let $\{ X_n \}$ be a Markov chain with transition matrix $P$, then $P$ is called  \emph{lumpable} for $\pi$ if the process $Y_n=\pi(X_n)$ is Markov for all choices of the initial distribution $p$. The necessary and sufficient conditions for lumpability are quite restrictive, as demonstrated by the following result:
\begin{theorem}{\cite{Kemeny60}} \label{thm:lumpable}
Suppose $P$ is an irreducible stochastic matrix, then $P$ lumpable with respect to $\pi:\Ao\to\At$ if and only if for any $y_1,y_2 \in \At$ we have
\begin{equation}
\label{eqt:lumpable}
\sum\limits_{x_2 \in \pi^{-1}(y_2)} P_{x_1 x_2} = \sum\limits_{x_2 \in \pi^{-1}(y_2)} P_{\tilde{x}_1 x_2}
\end{equation}
for any $x_1,\tilde{x}_1 \in \pi^{-1}(y_1)$. The transition matrix of the factor chain $\{Y_n=\pi(X_n)\}$
is then given by 
$$P^{(\pi)}_{y_1 y_2} = \sum\limits_{x_2 \in \pi^{-1}(y_2)} P_{x_1 x_2}.
$$
\end{theorem}
This condition is indeed very restrictive, in part due to a relatively strong requirement that the factor process $\{Y_n\}$ must be Markov for all initial distributions. 
Instead, one could require Markovianity only for a specific given initial distribution, this is a so-called weak lumpability property. It turns out that this question can be answered algorithmically in polynomial time \cite{Gurvits05}. Even though weak lumpability is a indeed a weaker condition than lumpability, it is still rather exceptional.

\subsection{Fully supported Markov chains}
\label{sec:fullypositive}
Sufficient conditions for a factor measure to be a $g$-measure are substantially less 
restrictive than the conditions for (weak-) lumpability. We will discuss some positive and negative results, starting with the very basic positive result for Markov chains with strictly positive transition matrices $P$. This case was first considered in \cite{Harris55} and comes with an estimate of the continuity rate of the conditional probabilities ($g$-functions) of the factor measure:
\begin{theorem}[\cite{Harris55}]
\label{thm:harris}
Let $\Pt$ be a one-block factor of a Markov measure  $\Po$ with a positive transition matrix $P$,
then $\Pt$ is a $g$-measure satisfying
\[ 
\mathrm{var}_n(g) = \mathcal{O}\left( c^n \right),
\]
for some $0<c<1$.
\end{theorem}

Let us only mention an intuitive, rough, argument for this result; suppose $P>0$ is the transition matrix of the Markov process $\{ X_n \}$. Suppose $y \in \Ot$ then, ignoring some technicalities, we can consider the behaviour of $\Po$ on $\Oo_y = \pi^{-1}(y)$. In particular, the transition from $X_{n+1}$ to $X_n$ in $\Oo_y$ will be given by a positive rectangular matrix. It is well known that, if this matrix is square, then the corresponding map between the distributions of $X_{n+1}$ and $X_n$ is a contraction. For a rectangular matrix we can obtain the same result by using the Hilbert projective metric on the relevant distribution spaces. It is easy to show that this contraction will be uniform in $n$ and therefore the result follows. A version of this argument can also be used to prove the more general results in \cites{Chazottes03,Yoo10,Pi1,Pi2}.

\subsection{Highly non-regular factor measure}
\label{sec:furstenberg}
A factor measure $\nu$ of a Markov measure $\mu$ is not necessarily a $g$-measure. This situation can arise when any version of the conditional probabilities has an essential discontinuity in at least one point of $\Ot$. In more extreme cases the conditional probabilities can be discontinuous everywhere. One such example was  discussed by Blackwell \cite{Blackwell56}, 
Furstenberg \cite[Theorem IV.6]{Furstenberg58}, Walters \cite{Walters86} and Lorinzi et al \cite{Maes}.
Let $(X_n)_{n \in \Z_+}$ be a Bernoulli process taking values in $\{-1,1\}$ with 
$$
\Po(X_n=1) = 1-\Po(X_n=-1) = p,
$$ 
for $0<p<1$, $p \neq \frac{1}{2}$. Then the process $(\tilde{X}_n)_{n \in \Z_+}$  with $\tilde{X}_n = (X_{n},X_{n+1})$ is Markov. 
Consider the factor process $Y_n = \pi(\tilde{X}_n) = \pi(X_{n},X_{n+1}) = X_n X_{n+1}$.
Thus  the factor  process $\{Y_n\}$ can be viewed as either a two-block factor of a Bernoulli (and hence, also Markov) process $\{X_n\}$, or a 1-block factor of an extended 
Markov process $\{\tilde X_n\}$, c.f., Remark \ref{remhigher} above.

 Note that $\Ot=\pi(\Oo)$ is the full shift on two symbols $\{-1,1\}$.
In this example the conditional probabilities of the factor process $\{Y_n\}$ are discontinuous everywhere. Indeed, it is easy to see that every fibre over $y\in \Ot$, i.e. $\Oo_y = \pi^{-1}(y) \subset \Oo$, consists of two points
$$
x^+_y=(1, y_0, y_0y_1, y_0y_1y_2,\ldots)\text{ and } x^{-}_y=(-1,-y_0,-y_0y_1,-y_0y_1y_2,\ldots).
$$
We can now explicitly compute the conditional probabilities:
\[
\begin{aligned}
\Pt(y_0|y_1^n) & = \frac{\Pt(y_0^n)}{\Pt(y_1^n)} = \frac{\Po((x_y^+)_0^{n+1})+\Po((x_y^-)_0^{n+1})}{\Po((x_y^+)_1^{n+1})+\Po((x_y^-)_1^{n+1})}.
\end{aligned}
\]
Since $\mu$ is the Bernoulli measure, it is easy to see that, with $S_n=\sum_{k=0}^n y_0y_1\ldots y_k$,
one has
\[\aligned
{\Pt(y_0^n)}&=p^{1+\frac{n+1+S_n}{2}}(1-p)^{\frac{n+1-S_n}{2}}+p^{\frac{n+1-S_n}{2}}(1-p)^{1+\frac{n+1+S_n}{2}}\\
&=p^{\frac{n+1}{2}}(1-p)^{\frac{n+1}{2}}\left[p\left(\frac{p}{1-p}\right)^{\frac{S_n}{2}}+(1-p)\left(\frac{1-p}{p}\right)^{\frac{S_n}{2}}\right].
\endaligned
\]
Similarly, 
\begin{eqnarray*}
\Pt(y_1^n)&=&p^{\frac{n}{2}}(1-p)^{\frac{n}{2}}\left[p\left(\frac{p}{1-p}\right)^{\frac{\widetilde{S_n}}{2}}+(1-p)\left(\frac{1-p}{p}\right)^{\frac{\widetilde{S_n}}{2}}\right],
\end{eqnarray*}
where $\widetilde{S_n}= \sum_{k=1}^{n} y_1 y_2 \cdots y_{k}$.
Since,
 $S_n=y_0(1+\widetilde{S_n})$, using $\lambda=p/(1-p)$, one has
$$\aligned
\Pt(y_0=1|y_1^n)&=\sqrt{p(1-p)}\left(\frac{p\lambda^{\frac{\widetilde{S_n}+1}{2}}+(1-p)\lambda^{-\frac{\widetilde{S_n}+1}{2}}}{p\lambda^{\frac{\widetilde{S_n}}{2}}+(1-p)\lambda^{-\frac{\widetilde{S_n}}{2}}}\right)
=\sqrt{p(1-p)}\left(\frac{p\sqrt{\lambda}\lambda^{\widetilde{S_n}}+\frac{(1-p)}{\sqrt\lambda}}{p\lambda^{\widetilde{S_n}}+(1-p)}\right)\\
&=:\frac{a\lambda^{\widetilde{S_n}}+b}{c\lambda^{\widetilde{S_n}}+d},
\endaligned
$$where 
$$
\frac ac= \sqrt{p(1-p)\lambda}=p \neq 1-p=\sqrt{\frac {p(1-p)}\lambda}=\frac bd,
$$
since $p\neq \frac{1}{2}$. Suppose for simplicity that $\lambda>1$. 
For any $y_1^n$, one can choose a continuation $z_{n+1}^{5n}$ such that $\widetilde{S_{5n}} \gg 0$. Equally well, one can choose a continuation $w_{n+1}^{5n}$ such that $\widetilde{S_{5n}} \ll 0$. In the first case, 
\[
\Pt(y_0=1|y_1^n z_{n+1}^{5n})\simeq \frac{a}{c} = p
\]
and in the second case,
\[
\Pt(y_0=1|y_1^n w_{n+1}^{5n})\simeq \frac{b}{d} = 1-p.
\]
Therefore, the conditional probabilities $\Pt(y_0=1|y_1y_2\cdots)$ are everywhere \emph{discontinuous}. In some sense this is the worst possible and most irregular behaviour
possible. At the same time, when $p = \frac{1}{2}$, $\Pt$ is a Bernoulli(\textonehalf,\textonehalf) product measure on $\{-1,1\}^{\Zp}$.
This example therefore highlights that regularity of the factor measure depends on both the properties of the coding map and the transition probabilities.

\subsection{Fibre mixing condition}
\label{sec:fibremixing}
In previous examples we saw that it was important to consider the behaviour of the Markov process $\{X_n\}$, given a realisation of the factor process $\{Y_n\}$. In particular the structure of the fibres $\pi^{-1}(y)$,  $y \in \Ot$, plays a crucial rule (c.f., Blackwell-Furstenberg example above). Theorem \ref{thm:harris} can also be interpreted in this way. To see this, recall that positivity of $P$ implies that transitions between any letters, consistent with the fibre, are allowed. The regularity of the factor process is a consequence of the fact that each transition in this fibre, described by a positive rectangular matrix, acts as a contraction on distributions. 
The most general sufficient condition \cite{Yoo10} for factors of Markov measures to be regular has a similar flavour. In particular, in \cite{Yoo10} the above idea is generalised from positive matrices to the analogon of \emph{primitive} matrices in the context of fibres; fibre mixing.

\begin{definition}[Fibre mixing]
Let $\Oo,\Ot$ be subshifts of finite type and $\pi:\Oo \to \Ot$ is a surjective 1-block factor.
We say that $\pi$ is \mybf{fibre mixing} if, for all $y \in \Ot$, for all $x,\tilde x \in \Oo_y$ and  every $n \in \Zp$, there exists an $\hat{x} \in \Oo_y$, such that $x_{0}^{n} = \hat x_{0}^{n}$ and $\tilde x^{\infty}_{n+m} = {\hat x}^{\infty}_{n+m}$, for some $m \in \Zp$.
\end{definition}
Indeed,  fibre mixing is a sufficient condition for the factor measure to be regular.
\begin{theorem}[Yoo \cite{Yoo10}]\label{thm:Yoo}
Suppose \begin{enumerate}
\item[(i)] $\pi:\Oo \rightarrow \Ot$  is a surjective 1-block factor map between irreducible subshifts of finite type
$\Oo$ and $\Ot$,
\item[(ii)] $P$ is an irreducible stochastic matrix, compatible with the SFT $\Oo$,
and $\Po$ is the corresponding stationary Markov
measure on $\Oo$.
\end{enumerate}
Suppose the factor $\pi$ is fibre mixing. Then $\Pt=\Po\circ\pi^{-1}$ is a $g$-measure on $\Ot$, for a H\"{o}lder continuous $g$-function.
\end{theorem}

This result provides the most general set of sufficient conditions for regularity of factors of Markov chains
known to date. Other sufficient conditions, e.g., found in \cites{Chazottes03,Kempton11} imply fibre mixing and are strictly stronger.
Let us reiterate that imposing conditions on fibres alone (i.e., the topological conditions
on $\Oo,\Ot$, and $\pi$) is not optimal: the necessary and sufficient conditions must also take $P$ into account, as demonstrated by the Blackwell-Furstenberg example discussed above. 

Yoo has also established the variant of  Theorem \ref{thm:Yoo} for factors of Gibbs measures with H\"older continuous potentials. Piraino \cite{Pi1,Pi2}  has provided an alternative proof in case of 
H\"older continuous potentials, and extended the result to classes
of potentials satisfying the Walters and Bowen conditions, showing in particular, that these
classes are preserved under 1-block factorization.
\section{Continuous measure disintegrations}
\label{sec:mainresults}
In this section we will argue that imposing conditions on the behaviour of 
conditional measures on the fibres provides a more appropriate framework
to study properties of the factor measures. As the first step, one has to properly define the conditional measures on the fibres.
Fortunately, general results of measure theory provide the necessary tools.
\begin{definition}
We call ${\bf \Po}_{\Ot} = \{ \Po_y \}_{y \in \Ot}$ a family of conditional measures for $\Po$ on the fibres $\Oo_y$ if $\Po_y$ is a Borel probability measure on the fibre $\Omega_y$,
\[
\Po_y(\Oo_y) = 1,
\]
for all $f \in L^1(\Oo,\Po)$ the map
\[
y \to \int_{\Oo_y} f(x) \Po_y(dx)
\]
is measurable and
\[
\int_\Oo f(x) \Po(dx) = \int_\Ot \int_{\Oo_y} f(x) \Po_y(dx) \Pt(dy).
\]
\end{definition}
We will also refer to a family of conditional measures ${\bf \Po}_{\Ot} = \{ \Po_y \}_{y \in \Ot}$ for $\Po$ on  fibres $\Oo_y$ as a disintegration of $\mu$ with respect to $\pi:\Oo\to\Ot$.

By a celebrated theorem of von Neumann, for all subshifts $\Oo,\Ot$, a given continuous surjection $\pi: \Oo \to\Ot$ and any Borel measure $\Po$ on $\Oo$, there exists  a disintegration ${\bf \Po}_{\Ot} = \{ \Po_y \}_{y \in \Ot}$ of $\Po$ with respect to $\pi$. Moreover, the disintegration is essentially unique in the sense that for any two disintegrations of $\Po$, $\{\Po_y\}$ and $\{\tilde\Po_y\}$, we have $\Pt(\{y:\tilde{\Po}_y(.) = \Po_y(.) \}) = 1$. We will be interested in continuous measure disintegrations (CMD): a measure disintegration
$\mathbf \Po_\Ot=\{\Po_y\}$ is called continuous if for every continuous function  $f :\Oo\to \R$, the function
\[
y \to \int_{\Oo_y} f(x) \Po_y(dx)
\]
is continuous. When a disintegration satisfies this constraint we call it a Continuous Measure Disintegration (CMD). Note that any measure $\mu$ admits at most one continuous disintegration.

As the conditional measures $\Po_y$ are not, in general, translation invariant, we introduce the following notation for cylinder sets in $\Oo_y$:
\[
{}_n[a_k^m] = \left\{ x \in \Oo_y : x_{n+k}^{n+m} = a_k^m \right\},
\]
for $a\in \Ot$ and $n,k,m \in \Zp$. 
Using  approach similar to that of \cite{Verbitskiy11}, we will now show that a measure disintegration can be used to find an expression for the conditional probabilities of a factor measure. 

\begin{theorem}
\label{thm:decomp}
Suppose \begin{enumerate}
\item[(i)] $\pi:\Oo \rightarrow \Ot$  is a surjective 1-block factor map between irreducible subshifts of finite type
$\Oo$ and $\Ot$,
\item[(ii)] $P$ is an irreducible stochastic matrix, compatible with the SFT $\Oo$, and $\Po$ is the corresponding stationary Markov
measure on $\Oo$.
\end{enumerate}
Suppose $\{\Po_y\}_{y\in \Ot}$ is a disintegration of $\mu$.
Then $\Pt=\Po\circ\pi^{-1}$ is consistent with the positive measurable normalized function $\tilde g:\Ot \to (0,1)$, i.e.,
$$
\Pt(y_0|y_1,y_2,\ldots) = \tilde g(y)\quad \Pt-a.e.,
$$
where 
\begin{equation}\label{tildeg}
\tilde{g}(y) =
\int_{\Oo_{Ty}} \left[  \sum\limits_{a\in\pi^{-1}y_0} \frac{p_a P_{a,x_0}}{p_{x_0}} \right] \Po_{Ty}(dx) 
= 
\sum_{a' \in \pi^{-1}y_1} \left[\sum\limits_{a \in \pi^{-1}y_0} \frac{p_aP_{a,a'}}{p_{a'}}\right] \Po_{Ty}({}_0[a']),
\end{equation}
and $
{}_0[ a'] = \left\{ w\in\Oo: w_{0}=a' \right\}$, $T:\Ot\to\Ot$ is the left shift on $\Ot$, for $y=(y_0,y_1,\ldots)$, $Ty=(y_1,y_2,\ldots)\in\Ot$.
\end{theorem}
\begin{proof}
The expression for $\tilde g$ originates from the following `finite-dimensional' equality:
denote by $\P$ the joint distribution of $(\{X_n\},\{Y_n\})$, where $\{X_n\}$ is the stationary Markov chain with the transition probability matrix $P$, and $Y_n=\pi(X_n)$ for all $n$. Then
$$\aligned
\P(y_0|y_1^n)&=\frac {\P(y_0y_1^n)}{\P(y_1^n)}=\frac {\sum\limits_{x_0^n\in \pi^{-1}y_0^n}
\P(x_0x_1^n)}{\P(y_1^n)}
=\sum_{x_1^n\in\pi^{-1}y_1^n} \left[\sum_{x_0\in \pi^{-1}y_0}  {\P(x_0|x_1^n) } \right] 
\frac {\P(x_1^n)}{\P(y_1^n)}\\
&=\sum_{x_1^n\in\pi^{-1}y_1^n} \left[\sum_{x_0\in \pi^{-1}y_0} \frac{p_{x_0}P_{x_0,x_1}}{p_{x_1}}  \right] 
\frac {\P(x_1^n)}{\P(y_1^n)} \\
& = \sum_{x_1 \in\pi^{-1}y_1} \left[\sum_{x_0\in \pi^{-1}y_0} \frac{p_{x_0}P_{x_0,x_1}}{p_{x_1}}  \right] 
\P\bigl(X_1 = x_1 |Y_1^n = y_1^n \bigr).
\endaligned
$$
The Markov measure $\Po$, corresponding to $\{ X_n \}$, is a $g$-measure for the function $g(x)=\frac{ p_{x_{0}} P_{x_0,x_1}   } {p_{x_1}}$, where $p$ is the invariant distribution: $pP=p$. We will now show that $\tilde g$, given by \eqref{tildeg}, is positive and normalized and finally that $\Pt=\Po\circ\pi^{-1}$ is consistent with $\tilde{g}$.

It is easy to check that $\tilde g$ is normalized. Indeed,
$$\aligned
\sum_{y_0\in\At} &\tilde g(y_0,y_1,y_2,\ldots) = \sum_{y_0\in\At}\left( \sum_{a' \in \pi^{-1}y_1} \left[\sum\limits_{a \in \pi^{-1}y_0} \frac{p_aP_{a,a'}}{p_{a'}}\right] \Po_{Ty}({}_0[a'])\right)\\
&= \sum_{a' \in \pi^{-1}y_1} \left[\sum_{y_0\in\At}
\sum\limits_{a \in \pi^{-1}y_0} \frac{p_aP_{a,a'}}{p_{a'}}\right] \Po_{Ty}({}_0[a'])=\sum_{a' \in \pi^{-1}y_1} \left[\sum_{a\in\Ao} \frac{p_aP_{a,a'}}{p_{a'}}\right] \Po_{Ty}({}_0[a'])\\
&=\sum_{a' \in \pi^{-1}y_1} 1 \cdot \Po_{Ty}({}_0[a'])=1,
\endaligned
$$
where we used that since $p$ is the invariant distribution: $pP=p$, or 
$\sum_{a\in\Ao} {p_aP_{a,a'}}={p_{a'}}$ for all $a'\in\Ao$, and hence $\tilde g$ is normalized.

The measurability of $\tilde g$  follows immediately from the measurability of the measure disintegration
$\{\mu_y\}$. The positivity of $\tilde g$ is readily checked as well.
Let $y=(y_0,y_1,\ldots)\in \Ot$, then the transition
from $y_0$ to $y_1$ is allowed in $\Ot$. Since $\pi:\Oo\to \Ot$ is surjective, it means that there is at least one pair $(a,a')$ such that $\pi(a)=y_0$, $\pi(a')=y_1$ and $P_{aa'}>0$.
Since the Markov chain is assumed to be irreducible it follows that the invariant distribution $p$ is strictly positive, and hence
$$
\varkappa = \min_{a,a':\ P_{aa'}>0} \frac{p_a P_{a,a'}}{p_{a'}}>0.
$$
Therefore, 
\begin{equation}\label{tildegpos}
\tilde{g}(y) =
\sum_{a' \in \pi^{-1}y_1} \left[\sum\limits_{a \in \pi^{-1}y_0} \frac{p_aP_{a,a'}}{p_{a'}}\right] \Po_{Ty}({}_0[a'])\ge \sum_{a' \in \pi^{-1}y_1}\varkappa \ \Po_{Ty}({}_0[a'])=\varkappa>0.
\end{equation}
Now we are going to show that $\Pt=\Po\circ\pi^{-1}$ is consistent with $\tilde g$, or, equivalently, that
for any continuous $h:\Ot\to \R$, one has
$$
\int_{\Ot} h(y) \Pt(dy) =\int_{\Ot}  \sum\limits_{b \in \At:\  by \in \Ot} h(by)  \tilde{g}(by) \Pt(dy).
$$

Now we show consistency of $\Pt$ with $\tilde{g}$ by using the fact that $\Po$ is a $g$-measure for 
$$g(x)=\frac {p_{x_0}P_{x_0x_1}}{p_{x_1}}.$$
Consider an arbitrary  $h \in C(\Ot)$, and let $\{ \Po_y \}$ be a measure disintegration for $\mu$ and $\pi$, then
\[
\begin{aligned}
\int_{\Ot} h(y) \Pt(dy) =& \int_{\Oo} (h \circ \pi)(x) \Po(dx) \\
 = & \int_{\Oo} \left[ \sum\limits_{a \in \Ao: P_{a x_0}>0 }\ (h \circ \pi)\left(a x_0^\infty\right) 
 g(ax_0^\infty) \right] \Po(dx) \\
  = & \int_{\Ot} \int_{\Oo_y} \left[ \sum_{\substack{b \in \At:\\b\pi(x) \in \Ot}}\ 
  \sum_{\substack{a \in \pi^{-1}b\\ P_{a,x_0}>0}} (h \circ \pi)(ax_0^\infty) g(ax_0) \right] \Po_y(dx) \Pt(dy) \\
  =&\int_{\Ot} \int_{\Oo_y} \left[ \sum_{\substack{b \in \At:\\by \in \Ot}}\ 
  \sum_{\substack{a \in \pi^{-1}b\\ P_{a,x_0}>0}} (h \circ \pi)(ax_0^\infty) g(ax_0) \right] \Po_y(dx) \Pt(dy) \\
  = & \int_{\Ot} \left( \sum\limits_{\substack{b \in \At:\\by \in \Ot}} h(by) \int_{\Oo_y} \left[  \sum\limits_{a \in \pi^{-1}b} g(ax_0) \right] \Po_y(dx) \right) \Pt(dy) \\
    = & \int_{\Ot}  \sum\limits_{\substack{b \in \At:\\by \in \Ot}} h(by)  \tilde{g}(by) \Pt(dy).
    \end{aligned}
\]
Thus, $\Pt$ is consistent with a positive normalized function $\tilde{g}:\Ot\to(0,1)$.
\end{proof}

Therefore, if for some disintegration $\Po_\Ot=\{\Po_y\}$, the function $\tilde{g}$, as defined in equation (\ref{tildeg}), is continuous, then $\nu$ is a $g$-measure.
There are two obvious sets of sufficient conditions for continiuty of $\tilde g$.

\begin{corollary}\label{pro:sufficientconditions}
Under conditions of Theorem \ref{thm:decomp}, the measure $\Pt$ is a $g$-measure if  there exists a disintegration $\Po_\Ot=\{\Po_y\}$ such that $
\tilde g(y)$, given by (\ref{tildeg}), is a continuous function on $\Ot$.

In particular, $\tilde g$ is continuous if one of the following conditions holds:
\begin{itemize}
\item[1)] matrix $Q=(Q_{a,a'})$ with $Q_{a,a'} = \frac{p_aP_{a,a'}}{p_{a'}} $ satisfies:
\begin{equation}\label{cond1}
\sum_{a \in \pi^{-1}(b)} Q_{a a'} = \sum_{a \in \pi^{-1}(b)} Q_{a a''},
\end{equation}
for any $b \in \At$ and any $a',a'' \in \pi^{-1}(b')$, where $b'\in \At$.
\item[2)] $\Po$ admits a \emph{continuous measure disintegration} on the fibres 
$\{\Oo_y=\pi^{-1}(y):  y\in \Ot\}$;
\end{itemize}
\end{corollary}
\begin{proof} If $\tilde g$ is indeed a continuous function, then $\nu$ is a $g$-measure by definition.
We only have to show that conditions (1) and (2) imply continuity of $\tilde g$.
Let us start with the first condition (\ref{cond1}). Since
\begin{equation}\label{tildegtwo}
\tilde{g}(y)=\sum_{a' \in \pi^{-1}y_1} \left[\sum\limits_{a \in \pi^{-1}y_0} \frac{p_aP_{a,a'}}{p_{a'}}\right] \Po_{Ty}({}_0[a'])= \sum_{a' \in \pi^{-1}y_1} \left[ \sum\limits_{a \in \pi^{-1}y_0} Q_{a,a'}\right] \Po_{Ty}({}_0[a']).
\end{equation}
Condition (\ref{cond1}) implies that for all $a'\in\pi^{-1}y_1$, the sums in the square-brackets
have the same value. Let us denote the common value by
$S_{y_0,y_1}$. Therefore,
$$
\tilde{g}(y)=S_{y_0,y_1}\sum_{a' \in \pi^{-1}y_1}  \Po_{Ty}({}_0[a'])= S_{y_0,y_1}
$$
since $\mu_{Ty}$ is a Borel probability measure on the fibre $\Oo_{Ty}=\Oo_{(y_1,y_2,\ldots)}$.

Let us now consider the second assumption: suppose $\Po$ admits a continuous measure disintegration on the fibres $\{ \Oo_y \}$, then for  any $f \in C(\Oo)$, $y \to \Po_y(f)=\int fd\Po_y$ is continuous. In particular, since for any $b\in\At$, the function
$$
 G^b(x) =\sum\limits_{a\in\pi^{-1}b} \frac{p_a P_{a,x_0}}{p_{x_0}}, 
$$
is  continuous on $\Oo$ as a function of $x$, we conclude that $\tilde{g}$ is continuous 
and hence $\Pt$ is a $g$-measure.
\end{proof}
\begin{remark}
The first condition is simply a standard (strong) lumpabililty condition for the time-reversal of the original Markov chain.
 Note that lumpabililty conditions for the chain and its reversal are not equivalent in general.
In this instance, however, we only consider the stationary chains, and hence, one should compare
the weak lumpability conditions for the chain and its time reversal. 
It is somewhat surprising that we finish with the strong lumpability condition for the reversed chain,
and not the weak lumpability condition.

 \end{remark}
\begin{remark} The second sufficient condition requires existence of continuous disintegration for $\Po$: i.e., continuity of the map
\begin{equation}\label{continuityint}
y\mapsto \int_{\Oo_y} f(x) \Po_y(dx)
\end{equation}
for every continuous $f$ on $\Oo$. However, we only need continuity of integrals of  rather `simple' functions of the form
\begin{equation}\label{continuityfunc}
G^b(x)= \sum\limits_{a\in\pi^{-1}b} \frac{p_a P_{a,x_0}}{p_{x_0}}, \quad b\in \At.
\end{equation} 
Thus the question is what is the relation between the requirements 
that there exists of a continuous measure disintegration for $\Po$,
and that there exists a disintegration such that for all $b\in\At$, the map $\Ot\ni y\mapsto \int_{\Oo_y} G^b(x)\Po_y(dx)\in\R_+$ is continuous. The first condition of Corollary \ref{pro:sufficientconditions}
then reads: for all $b\in\At$, $G^b(x)\equiv \text{const}$. In  the last section we present example
of an irreducible Markov chain such that $G^b(x)\equiv \text{const}$, but $\Po$ does not admit a continuous disintegration. However, in a 'non-trivial' case $G^b(x)\not\equiv\text{const}$, we believe the difference 
between requiring continuity $y\to\int_{\Po_y} f(x)\Po_y(dx)$ for all continuous $f$, versus, only for simple functions
depending only on the first coordinate $f(x)=f(x_0)$ is not substantial. The main reason is that we believe
that the general hypothesis on regularity of factor measures proposed in Statistical Mechanics \cite{Enter93}
applies to Markov chains as well.
\end{remark}

\begin{remark} In the following sections we will show that the second condition of Corollary \ref{pro:sufficientconditions} includes sufficient
conditions found earlier. More specifically, we will show that conditions of Theorem \ref{thm:Yoo}
imply existence of a continuous measure disintegration for $\mu$.
 
\end{remark}

We will proceed by investigating existence of a continuous measure disintegration using methods developed in thermodynamic formalism for fibred systems.

\section{Thermodynamic formalism for fibred systems}
\label{sec:fibreformalism}
There has been a lot of work done on thermodynamic formalism, equilibrium states and variational principles for fibred systems:
starting from the celebrated work of Ledrappier and Walters \cite{Ledrappier77} on relativized variational principles to the relatively comprehensive theory of Denker and Gordin \cite{Denker99}, as well as extensive
work on random subshifts of finite type \cite{Bogenschutz95}.
We apply the methods developed in this field to provide sufficient conditions for the existence of continuous fibre disintegrations of Markov measures. Moreover, we apply, for the first time in a dynamical setting, a method originating in Mathematical Statistics, developed by Tjur \cites{Tjur74,Tjur75} in the 1970's, which provides a \emph{constructive} approach to the construction of a continuous measure disintegration.

\subsection{Fibres as non-homogeneous subshifts of finite type}

The fibres of the factor map $\pi:\Oo\to\Ot$ are not translation invariant. However, 
they admit a nice topological description: namely, as  non-homogeneous or random subshifts of finite type.

\begin{definition}
Suppose $\mathbb S=\{ S_n \}_{n\ge 0}$ is a collection of non-empty finite sets of bounded size. Let
 $\Oo^{\mathbb S} = \prod\limits_{n\in \Z_+} S_n$ be the corresponding product space. 
 Assume also that we are given a sequence of $0/1$-matrices $\mathbb M=(M_n)_{n \in \Z_+}$, with $M_n$ indexed by $S_n \times S_{n+1}$, such that for each $n$, $M_n$ is \emph{reduced}:
it has no columns or rows with only $0$ entries. 
 
Then the set
\[
{\Oo_{\mathbb M}} = \left\{ x \in \Oo^{\mathbb S}: M_n(x_n,x_{n+1}) = 1\mathrm{,\ for\ all\ }n \in \Zp \right\},
\]
is called a non-homogeneous (random) subshift of finite type corresponding to the sequence $\mathbb M$.
\end{definition}

It is easy to see that if $\Oo=\Oo_M$ is a SFT, $\pi:\Oo\to\Ot$ is a 1-block factor map,
then for any $y \in \Ot$, the fibre $\Oo_y$ is  a non-homogeneous SFT: indeed,
let   $S_n^y = \pi^{-1}(y_n)$, and  put
$M_n^y(x_n,x_{n+1}) = 0 \Leftrightarrow M(x_n,x_{n+1})=0$ for all $n \in \Zp$ and $x_n \in S_n, x_{n+1} \in S_{n+1}$.  In other words, $\Oo_y=\Oo_{\mathbb M^y}$, for
$
\mathbb M^y=\{ M_n^y\}$, where $M_n^y$  is a submatrix of $M$ corresponding
to rows $\pi^{-1}(y_n)$ and  columns $\pi^{-1}(y_{n+1})$.

We recall the notion of a \emph{transitive} non-homogeneous subshift of finite type introduced by Fan and Pollicott \cite{Fan00}:
\begin{definition}
A non-homogeneous SFT $\Oo_{\mathbb M}$, corresponding to a reduced sequence of matrices $\mathbb M=(M_n)_{n \in \Z_+}$, is called \mybf{transitive}, i.e., there exists $m$ such that
\begin{equation}\label{transdef}
\prod_{j=n}^{n+m} M_j >0
\end{equation}
for all $n\ge 0$. The transitivity index of $\Oo_{\mathbb M}$ is defined as
the minimal $m$ such that 
(\ref{transdef}) holds for all $n\ge 0$.
\end{definition}

{ 
It turns out that the fibre mixing condition is equivalent to the requirement that  for every $y\in \Ot$ the fibre $\Oo_y$ is a transitive non-homogeneous SFT. Moreover, the transitivity index $m_y$ of $\Oo_y$ is bounded; in other words, there exists $m\ge 1$ such that the products

\begin{equation}\label{transdef2}
\prod_{j=n}^{n+m} M_j^{y} >0
\end{equation}
are positive for all $n\in \N$ and all $y\in \Ot$.
}

Let us start by recalling  a standard notion of allowed words in subshifts: if $X\subset {A}^{\Z}$ is a subshift,
then the word $(a_0,\ldots,a_k)\in A^{k+1}$ is called allowed in $X$, if there exists $x\in X$ such that
$x_0=a_0,\ x_1=a_1,\ldots,x_k=a_k$. The set of all allowed words of a subshift $X$ is denoted by $\mathcal L(X)$, and is called the language of $X$.

The following charaterization of fibre-mixing 1-block factor maps $\pi$ was given in \cite[Lemma 3.2]{Yoo}.

\begin{lemma}\label{lemmayoo} Suppose $\Oo$, $\Ot$ are mixing subshifts of finite type, and  $\pi: \Oo\rightarrow \Ot$ be a 1-block factor map. Then the following are equivalent:
\begin{enumerate}
\item $\pi$ is fiber-mixing, i.e., for every $x,\tilde x\in\Oo_y$ and all $n\in \N$, there exists $m\ge 0$ and $\bar x\in \Omega_y$ such that $$(x_0,\ldots, x_n)=(\bar x_0,\ldots, \bar x_n)\text{ and }
(\tilde x_{n+m},\tilde x_{n+m+1},\ldots)= (\bar x_{n+m},\bar x_{n+m+1},\ldots).
$$
\item  $\pi$ is sub-positive, i.e., there exists $k\in \N$ such that for any word $y_0^k=
(y_0,y_1,\ldots,y_k)\in \mathcal L(\Ot)$, any two allowed words $a_0^k, b_0^k\in \mathcal L(\Oo)$ such that 
$\pi(a_0^k)=\pi(b_0^k)=y_0^k$, there exists a third allowed word $c_0^k\in \mathcal L(\Oo)$ satisfying
$$
\pi(c_0^k) = y_0^k,\quad c_0=a_0,\quad{and}\quad c_k=b_k.
$$\end{enumerate}
\end{lemma}

A simple corollary of this result is the following:

\begin{corollary} 
\label{lem:lemma_mixing_primitive}
Under conditions of Lemma \ref{lemmayoo}, a 1-block factor map $\pi: \Oo\rightarrow \Ot$ is fiber-mixing if and only if for each $y\in\Ot$, the fibre  $\Oo_y$ is a transitive non-homogenous SFT.\end{corollary}
\begin{proof}
If $\pi$ is fibre-mixing, then $\pi$ is sub-positive for some integer $k$. Consider an arbitrary point $y\in\Ot$  and an integer $n$. Suppose that the 
 $|S_n|\times |S_{n+k}|$ matrix
\begin{equation}\label{prodmat}
M_{n,k-1}^{(y)}:=\prod_{j=n}^{n+k-1} M^{(y)}_j,
\end{equation}
is not strictly positive.
Thus there exist $a_{n}\in S_n=\pi^{-1}(y_n)$ and $b_{n+k}\in S_{n+k}=\pi^{-1}(y_{n+k})$ such that
the  $(a_n,b_{n+k})$-element of $M_{n,k-1}^{(y)}$ is zero. Informally, that means that we cannot connect
$a_{n}$ and $b_{n+k}$ by an allowed path in $\Omega_y$ of length $k+1$. On the other hand,
we can extend $a_n$ and $b_{n+k}$ to configurations $x,\tilde x\in \Omega_y$ such that $x_n=a_n$ and $\tilde x_{n+k}=b_{n+k}$. Therefore, if we consider two words of length $k+1$, namely, $x_{n}^{n+k}$ and $\tilde x_{n}^{n+k}$,
then $\pi(x_{n}^{n+k})=\pi(\tilde x_{n}^{n+k})=y_{n}^{n+k}$. Since $\pi$ is sub-positive, there exists
a third word
$\bar x_{n}^{n+k}$  such that $\pi(\bar x_{n}^{n+k})=y_{n}^{n+k}$ and $\bar x_{n}=x_n=a_n$, $\bar x_{n+k}=\tilde x_{n+k}=b_{n+k}$. Therefore, we arrived to a contradiction with the assumption that $M_{n,k-1}^{(y)}(a_n,b_{n+k})=0$. Thus all products of the form (\ref{prodmat}) are strictly positive. In particular,
this implies that for each fibre $\Omega_y$, the transitivity index $m_y$ is bounded by $k$ from above.

In the opposite direction, assume that each fibre is a transitive non-homogeneous subshift of finite type.
Suppose $x,\tilde x \in \Oo_y$.
Since $\Oo_y$ is  assumed to be  transitive, then for any $i\in\Z$, we have $\prod_{n=i}^{i+m(i)} M_n^{(y)} > 0$ for some finite $m(i)$, and therefore there exists an $\hat{x} \in \Oo_y$  with $\hat{x}_i = x_i$, $\hat{x}_{i+m(i)+1} = \bar{x}_{i+m(i)+1}$
such that
\[
M^{(y)}_i(\hat{x}_i,\hat{x}_{i+1}) \dotsb M^{(y)}_{i+m(i)}(\hat{x}_{i+m(i)},\hat{x}_{i+m(i)+1})>0.\]
However, this means that $x_0^{i} \hat{x}_{i+1}^{i+m} \bar{x}_{i+m+1}^\infty \in \Oo_y$m and therefore $\Oo_y$ is fibre mixing.
\end{proof}

\subsection{Non-homogeneous equilibrium states}
\label{sec:nonhomogeneouseqstates}
Now we are ready to apply methods of thermodynamic formalism to construct directly the conditional measures on the fibres.
The first and the most direct method is to use the approach of \cite{Verbitskiy11} which relies on the fundamental results of Fan and Pollicott \cite{Fan00} for transitive non-homogenous subshifts of finite type. Since the proof in the Markov case considered in the present paper is almost identical to (and, in fact, simpler than) the proof in the case of fully supported
$g$-measures in \cite{Verbitskiy11}, we will only sketch the necessary steps.
We start by introducing the \emph{averaging operators} acting on spaces of continuous functions on fibres $\Oo_y$:
$$
P_n^y f(x) = \sum_{\substack{a_0^n\in\pi^{-1}y_0^n:\\ a_0^nx_{n+1}^{+\infty}\in \Oo_y} }G_n^y(a_0\ldots a_nx_{n+1}\ldots) f(a_0\ldots a_nx_{n+1}\ldots),
$$
where  $G_n^y(x)$ is defined on $\Oo_y$ by
\begin{equation}\label{potGN}
 G_n^y(x) =\frac {  Q(x_0,x_{1})\ldots Q(x_n,x_{n+1})}
 {\sum\limits_{\substack{a_0^n:\ a_0^nx_{n+1}^{+\infty}\in \Oo_y}}
 Q(a_0,a_{1})\ldots Q(a_n,x_{n+1})},\quad Q(a,a') =\frac {p_a P_{a,a'}} {p_{a'}}, \ a,a'\in\Ao.
\end{equation}
Note that $\sum_{a_0^n:\ a_0^nx_{n+1}^{+\infty}\in \Oo_y}G_n^y(a_0\ldots a_nx_{n+1}\ldots)=1$ for all $x\in\Oo_y$, and hence $P_n^y \mathbf 1 = \mathbf 1$. A probability measure $\Po^y$ on $\Oo_y$ is called a non-homogeneous equilibrium state associated to $\mathbb G^y=\{ G_n^y\}$ if 
$$
\int_{\Oo_y} P_n^y f(x) \mu^y(dx) = \int_{\Oo_y}  f(x) \mu^y(dx)
$$
for all $f\in C(\Oo_y)$ and every $n\in\N$. Next we will show that the equilibrium states $\Po^y$ form a continuous measure disintegration, for now, we use a superscript to distinguish from the notation for a disintegration.
\\
The sequence of $\mathbb G^y=\{ G_n^y\}$ given by (\ref{potGN}), can easily be seen
to satisfy the conditions of Theorem 1 of \cite{Fan00}, and we immediately get the following corollary:

\begin{corollary}  Suppose $\Oo, \Ot$ are irreducible SFT's, and a 1-block surjective factor map $\pi:\Oo\to\Ot$ is such that
 $\Oo_y$ is a transitive non-homogenous SFT for every $y\in\Ot$. Then
 for each $y\in\Ot$ there exists a unique non-homogeneous equilibrium state $\mu^y$ associated to $\mathbb G^y=\{ G_n^y\}$. Moreover,  
\begin{equation}\label{convergencefibre}
P_n^y f(x) \rightarrow \int_{\Oo_y} f(x) \mu^y(dx)
\end{equation}
 uniformly on $\Oo_y$, as $n\to\infty$.
 \end{corollary}
Furthermore, the convergence in (\ref{convergencefibre}) turns out to be uniform in $y$ as well. Using this rather strong property ,we also immediately get the following corollary of Lemma 3.4 and 3.5 \cite{Verbitskiy11}:
 
 \begin{proposition}\label{integraleq}
Under the above conditions, the family $\{\Po^y\}$ of non-homogeneous equilibrium states on $\Oo_y$
associated to $\mathbb G^y$ forms a disintegration of $\mu$, i.e., for every continuous function $f$
one has
$$
\int_{\Oo} f(x)\mu(dx) =\int_{\Ot} \int_{\Oo_y} f(x)\Po^y(dx) \Pt(dy).
$$
Moreover, the  family $\{\Po^y\}$ is in fact continuous: for every continuous $f$,
$$
y\mapsto  \int_{\Oo_y} f(x)\Po^y(dx)
$$
is a continuous function on $\Pt$.
\end{proposition}

Therefore, by Proposition \ref{pro:sufficientconditions}, we conclude that $\Pt$ is a $g$-measure.

\begin{remark}
The above method can be summarized as follows. The conditional measures on fibres are equilibrium states for the same potential as the starting measure $\Po$. One needs to establish uniqueness of equilibrium states on the fibres first, and then prove continuity of the resulting family. In this particular case, one obtains continuity from the double uniform convergence of the averaging (transfer) operators. In the following section, we are going to show that uniqueness on each fibre is in fact sufficient, and one obtains continuity effectively for free.
\end{remark}

\subsection{Constructive approach to conditioning on fibres}
\label{sec:tjurformalism}
General results on the existence of measure disintegrations are not constructive. To alleviate this problem, Tjur  \cites{Tjur74,Tjur75} proposed  a more direct method: the conditional measures $\Po_y$ on fibres can be obtained directly, in a unique way, as a limit of measures conditioned on sets with positive measure around $y$.

Suppose $y\in \Ot$ and let $D_{y}$ be the set of pairs $(V,B)$, where $V$ is an open neighbourhood of $y$ and $B$ is a subset of $V$ with positive measure:
\[
D_{y} = \left\{ (B,V): V \mathrm{\ open,\ }y \in V,\ B \subset V,\  \Pt(B) > 0 \right\}.
\]
Now equip the collection $D_{y}$ with a partial order given by $(V_1,B_1)\succcurlyeq (V_2,B_2)$, if $V_1\subseteq V_2$. This partial order is upwards directed, as, for any $(V_1,B_1),(V_2,B_2) \in D_{y_0}$, there exists an element $(V_3,B_3)\in D_{y_0}$ such that $(V_3,B_3)\succcurlyeq (V_1,B_1)$ and
 $(V_3,B_3)\succcurlyeq (V_2,B_2)$. For each $(V,B)\in D_{y}$ we define a conditional measure $\Po^B$:
$$
\mu^B(\cdot)=\mu( \cdot\, |\, \pi^{-1}B).
$$
Since $D_{y}$ is upwards directed, the collection of conditional measures 
\[
\net_y = \left\{ \Po^B(\cdot):(V,B) \in D_y \right\},
\]
is a \mybf{net}, or a \mybf{generalized sequence}, in the space of probability measures on $\Oo$. We can now define the limit or accumulation points of this net as follows:
\begin{definition}
We call a measure $\tilde{\Po}$ on $\Oo$ an accumulation point of the net $\net_y$ if there exists a sequence $\{(V_n,B_n)\}_{n\ge 1} \subset D_y$, $n\ge 1$, such that
\[
\Po^{B_n}=\Po(\cdot|\pi^{-1}B_n) \to \tilde{\Po}, \mathrm{\ as\ }n \to \infty,
\]
weakly. Denote the set of all possible accumulation points by $\tjur_y$.
\end{definition}
By standard compactness arguments we immediately conclude that 
$\tjur_y \neq \varnothing$, and for each $\lambda_y\in \tjur_y$, one has
$\lambda_y(\Oo_y)=1$.
\begin{definition} The point $y\in\Ot$ is called a Tjur point if $\tjur_y$ is a singleton, i.e., the net $\net_y$ has a limit, which we denote by $\Po^y$.
\end{definition}
Two basic theorems by Tjur provide sufficient conditions for the existence of continuous measure disintegrations. The first theorem states that, when conditional measures $\Po^y$ are defined $\Pt$-almost everywhere, they form a measure disintegration.

\begin{theorem}\label{thm:tjur1}{\cite[Theorem 5.1]{Tjur75}} Suppose $\pi:\Oo\to\Ot$ is a continuous surjection, as defined above, and $\Pt=\Po\circ\pi^{-1}$. Assume, furthermore, that $\Pt$-almost all $y \in \Ot$ are Tjur points. Then, for any $f\in L^1(\Oo,\Po)$, $f$ is $\Po^y$-integrable for $\Pt$-almost all $y$,  and the function $y\mapsto \int f d\Po^y$ is $\Pt$-integrable and
$$
\int_\Oo f(x) \Po(dx) =\int_{\Ot} \Bigl[\int_{\Oo_y} f(x)\mu^y(dx)\Bigr]\Pt(dy).
$$
\end{theorem}
The second theorem  provides the desired continuity for the map $ y \mapsto \Po^y$.
\begin{theorem}\label{thm:tjur2}{\cite[Theorem 4.1]{Tjur75}} Denote by $\Ot_0$ the set of all Tjur points in $\Ot$.
Then the map 
$$
 y \mapsto \Po^y
$$
is continuous on $\Ot_0$.
\end{theorem}
As a corollary, we immediately conclude 
\begin{corollary}\label{mainTjur}
If $\Pt = \Po \circ \pi^{-1}$ and for all $y \in \Ot$ we have
$|\tjur_y| = 1$, i.e., 
all points are Tjur, then
$\mu$ admits a continuous disintegration, and hence $\nu$ is a $g$-measure.
\end{corollary}

\subsection{Gibbs measures on fibres}
\label{sec:gibbsformalism}
The main result of the  previous section states that existence  of a unique   limit of the sequence of conditional measures $\mu( \cdot |\pi^{-1}B)$, $B\searrow y$, for all $y \in \Ot$, is sufficient for the regularity of $\Pt$. However, this condition is not easy to validate directly. The general principle  for renormalization of Gibbs random fields formulated by van Enter, Fernandez, and Sokal, in the seminal paper \cite{Enter93}, states that the conditional measures must be Gibbs for the \emph{original} potential. Since the original measure $\Po$ is Markov, i.e., Gibbs for a two--point interaction,
we have to study the Gibbs Markov measures on the fibres. In the setting of this paper that means that the conditional measures are Markov. In fact, we have already seen this indirectly in the Fan-Pollicott construction on non-homogeneous equilibrium states on fibres. In this section we define Gibbs Markov measures on fibres and show the absence of phase transitions, i.e., prove uniqueness  on each fibre. In the following section we show that any limit measure in $\tjur_y$ must be Markov and, given that there is only one Markov measure on each fibre, we conclude that $|\tjur_y|=1$ for all $y \in \Ot$.
Suppose $\Oo,\Ot$ and $\pi:\Oo\to\Ot$ are defined as above and $\Po$ is a stationary Markov measure with $\Oo$ as its support.
\begin{definition} A Borel probability measure $\rho$ on $\Oo_y$ is called Gibbs Markov for the (irreducible) stochastic matrix $P$, if for all $n$ and $\rho$-almost all $x=(x_0,x_1,\ldots)\in \Oo_y$ 
\begin{equation}\label{defspec}
\rho( x_0^n |x_{n+1}^\infty) =\frac {  Q(x_0,x_{1})\ldots Q(x_n,x_{n+1})}
 {\sum\limits_{\substack{a_0^n:\ a_0^nx_{n+1}^{+\infty}\in \Oo_y}}
 Q(a_0,a_{1})\ldots Q(a_n,x_{n+1})},\quad Q(a,a') =\frac {p_a P_{a,a'}} {p_{a'}}, \ a,a'\in\Ao.
\end{equation}
\end{definition}
If we define the  \emph{interaction}  $\Phi=\{\Phi_{\Lambda}(\cdot)\}$ -- collection of functions
indexed by 
finite subsets $\Lambda$ of $\Zp$, by 
$$
\Phi_\Lambda(x)=\begin{cases} -\log Q(x_k,x_{k+1}),& \text{ if } \Lambda=\{k,k+1\},\\
0,&\text{ otherwise},
\end{cases}
$$
then the expression  (\ref{defspec}) can be rewritten in a more traditional Gibbsian form:
\begin{equation}\label{specgamma}
\rho( x_0^n |x_{n+1}^\infty) = \frac 1{Z_{[0,n]}(x_{n+1}^\infty)} \exp\left( -H_{[0,n]}(x)\right),
\quad H_{[0,n]}(x)=\sum_{\Lambda\cap[0,n]\ne\varnothing} \Phi_\Lambda(x),
\end{equation}
and $Z_{[0,n]}(x_{n+1}^\infty)=\sum_{a_0^n: a_0^nx_{n+1}^\infty\in \Ot_y}\exp\left( -H_{[0,n]}(a_0^nx_{n+1}^\infty)\right)$ is the corresponding partition function. We denote by $\gibbs_{\Oo_y}(\Phi)$ the set of all Gibbs probability  measures for the interaction $\Phi$. Since $\Oo_y$ is a non-homogeneous subshift of finite type, i.e., the \emph{lattice system} as described in \cite{Ruelle78}, the standard theory of Gibbs states implies $\gibbs_{\Oo_y}(\Phi)$ is a non-empty convex set of measures.
Moreover, the extremal measures are tail-trivial. Thus two extremal measures in $\gibbs_{\Oo_y}(\Phi)$ are either singular or equal. 

To prove uniqueness of Gibbs Markov measures on fibres we 
will use the classical  boundary uniformity  condition \cites{Fernandez05,Georgii74}.
Denote the right hand side of (\ref{specgamma}) by $\gamma_{[0,n]}(x_0^n|x_{n+1}^\infty)$, and for a continuous function $f$, let
$$
(\gamma_{[0,n]} f)(x) = \sum_{a_0^n: a_0^nx_{n+1}^\infty\in\Oo_y} f(a_0^nx_{n+1}^\infty)
\gamma_{[0,n]}(a_0^n|x_{n+1}^\infty).
$$
Then 
$\rho\in \gibbs_{\Oo_y}(\Phi)$ if and only if for every continuous $f$ on $\Oo_y$ the Dobrushin-Lanford-Ruelle equations
are valid for every $n\ge 0$
$$
\int_{\Oo_y} f(x) \rho(dx)= \int_{\Oo_y} (\gamma_{[0,n]} f)(x)\rho(dx).
$$
Given the fact that the non-homogeneous subshift of finite type $\Oo_y$ is transitive, 
$\Phi$ is a \emph{finite range} potential, it is easy to check that 
the family of probability kernels $\gamma_{[0,n]}(\cdot |x_{n+1}^\infty)$ satisfies the so-called boundary uniformity condition: there exists $c>0$ such that  for any $a_0^m\in\pi^{-1}(y_0^m)$, 
and every $x,\tilde x\in\Oo_y$, for all sufficiently large $n$, one has
\begin{equation}\label{BBC}
\bigl(\gamma_{[0,n]}\mathbbm 1_{[a_0^m]}\bigr)(x)\ge c\bigl(\gamma_{[0,n]}\mathbbm 1_{[a_0^m]}\bigr)(\tilde x).
\end{equation}
Applying  standard arguments for uniqueness of Gibbs measures under the boundary uniformity condition \cite{Fernandez05} one gets:
\begin{lemma}\label{lemmaUniq} Suppose $\Oo_y$ is a transitive non-homogeneous subshift of finite type, and the potential $\Phi$ is such that the family of probability kernels $\{ \gamma_{[0,n]}\}$ satisfies \eqref{BBC}. Then there
exists a unique Gibbs measure for $\Phi$ on $\Oo_y$, i.e., $|\gibbs_{\Oo_y}(\Phi)|=1$.
\end{lemma}
\begin{proof}
Consider two arbitrary extremal Gibbs measures $\rho,\tilde\rho\in \gibbs_{\Oo_y}(\Phi)$.
By integrating (\ref{BBC}) first with respect to $\rho(dx)$, and then with respect to $\tilde\rho(d\tilde x)$, one concludes that
\[
\begin{aligned}
\rho([a_0^m]) & = \iint \bigl(\gamma_{[0,n]}\mathbbm 1_{[a_0^m]}\bigr)(x) \rho(dx)\tilde\rho(d\tilde x)  \geq \iint c  \bigl(\gamma_{[0,n]}\mathbbm 1_{[a_0^m]}\bigr)(\tilde x) \tilde\rho(d\tilde x) \rho(dx) = c \tilde\rho([a_0^m]),
\end{aligned}
\]
and hence, $\rho\ge c\tilde \rho$. Similarly, $\tilde\rho\ge c\rho$. Since the distinct extremal measures in $\gibbs_{\Oo_y}(\Phi)$ must be singular, we conclude that $\rho=\tilde\rho$. Hence, $
\gibbs_{\Oo_y}(\Phi)$ has a unique extremal element, and therefore $
\gibbs_{\Oo_y}(\Phi)$ is a singleton.
\end{proof}

\subsection{Conditional measures are Markov}
\label{sec:conditionalmeasuresaremarkov}
We are now going to show that any limit point of the net $\net_y$ must be a Gibbs Markov measure on $\Oo_y$, i.e., $\tjur \subseteq \gibbs_{\Oo_y}(\Phi)$. Since we have already shown that $|\gibbs_{\Oo_y}(\Phi)|=1$ for all $y \in \Ot$, we conclude that $|\tjur_y|=1$ for all $y$, i.e., all points in $\Pt$ are Tjur, and hence the $\Pt$ is a $g$-measure.
\begin{proposition} Let $\Po$ be a stationary irreducible Markov measure for the interaction $\Phi = \{ \phi_{i,i+1}\}$ and let $\tjur_y$ be defined as above. For all $y\in\Ot$, one has $\tjur_y\subseteq \gibbs_{\Oo_y}(\Phi)$.
\end{proposition}
\begin{proof} Suppose $\rho\in \tjur_y$:
$$
\rho=\lim_{m \to \infty} \Po(\cdot |\pi^{-1} B_m),
$$
for some sequence $(V_{m},B_{m})\in D_y$. Without loss of generality we may assume $V_m=[y_0^m]$. Moreover, since any measurable set $B_m$ can be approximated arbitrarily well by cylinders, it is sufficient to consider only limit points of $\{\Po(\cdot|\pi^{-1}[y_{0}^{m}z_{m+1}^{m+n}])\}_{m,n\ge 0}$, provided $\Pt([ y_{0}^{m}z_{m+1}^{m+n}])>0$.
Denote the set of all limits points of such conditional measures by $\cyltjur_y$. We first prove the following lemma:
\begin{lemma}\label{TjurLimits}
For all $y \in \Ot$, any limit point in $\tjur_y$ is a linear combination of the limit points
in $\cyltjur_y$. 
\end{lemma}
\begin{proof}
Let $y \in \Ot$, $\lambda \in \tjur_y$ and $(B_m,V_m) \in D_y$ is a sequence such that $\Po^{B_m} \to \lambda$. 
It suffices to show that each $\Po^{B_m}$ is a limit point of linear combinations in $\cyltjur_y$.
For any $m,n \in \N$ we can find a collection $\{C^{(m)}_{n,l}\}$ of disjoint cylinder sets in $\Ot$, indexed by a finite set $L_{m,n}$, such that $\Pt\bigl(B_m \Delta \cup_{l \in L_{m,n}} C^{(m)}_{n,l}\bigr) <2^{-n}\Pt(B_m)$. 

Given any measurable set $A$, we have that $\left|\Po^{B_m}(A) - \Po^{\cup_{l \in L_{m,n}} C^{(m)}_{n,l}}(A)\right| \to 0$ as $n \to \infty$. Also note that
\[
\begin{aligned}
\Po^{\cup_{l \in L_{m,n}} C^{(m)}_{n,l}}(A) & = \frac{\Po(A \cap \pi^{-1} \cup_{l \in L_m} C^{(m)}_{n,l})}{\Po( \pi^{-1} \cup_{l \in L_m} C^{(m)}_{n,l})}  = \sum_{l \in L_{m,n}} \Po(A| \pi^{-1} C^{(m)}_{n,l}) \frac{ \Po( \pi^{-1} C^{(m)}_{n,l})}{\sum_{\tilde{l} \in L_{m,n}}\Po( \pi^{-1} C^{(m)}_{n,\tilde{l}})}.
\end{aligned}
\]
In other words, each $\Po^{B_m}$ is a limit point of linear combinations of measures of the form $\Po^{C_{n,l}^{(m)}}$. Therefore $\lambda$ is a limit point of linear combinations of measures in $\cyltjur_y$. 
\end{proof}
Hence, if we are able to prove that $\cyltjur_y \subseteq \gibbs_{\Oo_y}(\Phi)$, then
we are able to conclude that $\tjur_y \subseteq \gibbs_{\Oo_y}(\Phi)$ as well. 
Suppose 
$$
\rho=\lim_{m} \rho_m,\quad \rho_m=\mu(\cdot |\pi^{-1} [y_{0}^{m}z_{(m)}]),
$$
where $z_{(m)}$ is some finite word in alphabet $\At$, such that $\Pt([y_{0}^{m}z_{(m)}])>0$ for all $m$.
 We are going to show that $\rho$ is a Markov measure on $\Oo_y$, in other words
\begin{equation}\label{Markovrho}
\rho(x_0^n|x_{n+1}^{n+\ell}) = \rho(x_0^n|x_{n+1}) 
\end{equation}
for all $n\ge 0$, $\ell\ge 1$, and $x\in\Oo_y$. Since $\rho$ is the weak limit of $\rho_m$'s, it is thus
sufficient to establish
(\ref{Markovrho}) for $\rho_m$ for all sufficiently large $m$.

Consider $x\in\Oo_y$, fix $n\ge 0$, $\ell\ge 1$.  Choose $m_0$ such that  for all $m\ge m_0$, $K_m$ -- the length of the word $y_{0}^{m}z_{(m)}$, satisfies $K_m> n+\ell$; e.g., $m_0=n+\ell+1$ suffices. Then
\[
\begin{aligned}
\rho_m(x_0^n|x_{n+1}^{n+\ell})& = 
\frac{ \rho_m([x_0^n,x_{n+1}^{n+\ell} ])}
{\rho_m([x_{n+1}^{n+\ell}])}
=\frac{ \Po([x_0^n,x_{n+1}^{n+\ell}]\cap\pi^{-1}[y_0^mz_{(m)} ])}
{\Po([x_{n+1}^{n+\ell}]\cap\pi^{-1}[y_0^mz_{(m)} ])}
=\frac{ \sum\limits_{a_0^{K_m}\in \pi^{-1}[y_0^mz_{(m)} ]:  \ a_0^{n+\ell}=x_0^{n+\ell}} \Po(a_0^{K_m})}{\sum\limits_{b_0^{K_m}\in \pi^{-1}[y_0^mz_{(m)} ]:  \ b_{n+1}^{n+\ell}=x_{n+1}^{n+\ell}} \Po(b_0^{K_m})}\\
& = \frac{ \sum\limits_{a_0^{K_m}\in \pi^{-1}[y_0^mz_{(m)} ]:  \ a_0^{n+\ell}=x_0^{n+\ell}} \Po(a_0^{n}|a_{n+1}^{K_m}) \Po(a_{n+1}^{K_m})}{\sum\limits_{b_0^{K_m}\in \pi^{-1}[y_0^mz_{(m)} ]:  \  b_{n+1}^{n+\ell}=x_{n+1}^{n+\ell}} \Po(b_0^{n}|b_{n+1}^{K_m})
\Po(b_{n+1}^{K_m})}\\
& = \frac{ \sum\limits_{a_0^{K_m}\in \pi^{-1}[y_0^mz_{(m)} ]:  \ a_0^{n+\ell}=x_0^{n+\ell}} 
\Po( x_0^{n}|  x_{n+1}) \Po(a_{n+1}^{K_m})}{\sum\limits_{b_0^{K_m}\in \pi^{-1}[y_0^mz_{(m)} ]:  
\  b_{n+1}^{n+\ell}=x_{n+1}^{n+\ell}} \Po(b_0^{n}| x_{n+1})
\Po(b_{n+1}^{K_m})}\quad (\text{since $\mu$ is Markov})\\
& = 
\frac{ \Po( x_0^{n}|  x_{n+1}) \sum\limits_{a_0^{K_m}\in \pi^{-1}[y_0^mz_{(m)} ]:  \ a_0^{n+\ell}=x_0^{n+\ell}} 
 \Po(a_{n+1}^{K_m})
 }
 {
  \sum\limits_{b_0^{n} : \pi(b_0^n x_{n+1}) = y_0^{n+1},  
\  P_{b_{n} x_{n+1}}>0}
\Po(b_0^{n}| x_{n+1}) 
  \sum\limits_{b_{n+1}^{K_m}\in \pi^{-1}[y_{n+1}^mz_{(m)} ]:  
\  b^{n+\ell}_{n+1}=x^{n+\ell}_{n+1}} 
\Po(b_{n+1}^{K_m})
}
\\
& = \frac{\Po(x_0^n|x_{n+1})}{\sum\limits_{b_0^n : \pi(b_0^n x_{n+1}) = y_0^{n+1},  
\  P_{b_{n} x_{n+1}}>0} \Po(b_0^n|x_{n+1})},
\end{aligned}
\]
is independent of $m$ and of $x_{n+2}^{n+\ell}$. Hence, $\rho$, which is the weak limit of $\rho_m$'s
satisfies (\ref{Markovrho}), and is thus a Markov measure on $\Oo_y$.

\end{proof}

These results can now be used to show that fibre mixing does indeed imply existence of a continuous measure disintegration, and hence by Corollary \ref{mainTjur} regularity of the factor measure $\Pt$.
\begin{corollary}
Let $\Oo \subset \Ao^\Zp$ and $\Ot \subset \At^\Zp$ be mixing subshifts of finite type, and $\pi:\Oo \to \Ot$ a $1$-block factor map which is \mybf{fibre mixing}.
Suppose  $\Po$ is the stationary Markov measure consistent with $\Oo$, then $\Po$ admits a continuous measure disintegration and hence $\Pt=\Po\circ\pi^{-1}$ is a $g$-measure.
\end{corollary}

\begin{proof} By Lemma \ref{lemmaUniq}, for every $y\in\Ot$, there is a unique Gibbs Markov measure  on $\Oo_y$: $|\gibbs_{\Oo_y}(\Phi)|=1$. By Proposition \ref{TjurLimits}, $\varnothing\ne \tjur_y\subset \gibbs_{\Oo_y}(\Phi)$, and hence $|\tjur_y|=1$ for all $y\in\Ot$. Thus all points in $\Ot$ are Tjur,
and hence by Corollary \ref{mainTjur}, $\mu$ admits a continuous disintegration, which allows
us to conclude that $\nu$ is a $g$-measure.
\end{proof}

\section{Examples}
\label{sec:examples}
Existence of a continuous measure disintegrations of Markov measures
thus follows  from the fibre-mixing condition. In fact, it is a weaker condition:
it implies regularity of the Furstenberg example (see Section \ref{sec:furstenberg}) for the exceptional parameter value $p=\frac{1}{2}$, which is not fibre mixing. Recall, $\{X_n\}_{n \in \Zp}$ is a Bernoulli process, with a parameter $p \in (0,1)$ taking values in $\Ao=\{-1,1\}$ and $\{Y_n\}_{n\in\Zp}$ is defined by $Y_n=X_{n}X_{n+1}$. The fibres in this example are $\Oo_y = \pi^{-1}(y)=\left\{ x_y^+, x_y^- \right\}$, where
\[
\begin{aligned} 
x_y^+ &= (1,\ y_0,\ y_0\cdot y_1,\ y_0\cdot y_1\cdot y_2,...),\quad
x_y^- &= (-1,\ -y_0,\ - y_0\cdot y_1, \ -y_0\cdot y_1\cdot y_2,\ldots).
\end{aligned}
\]

If $p = \frac{1}{2}$, then $\{Y_n\}$ are independent, and $\Pt=\Po\circ\pi^{-1}$ is the  Bernoulli
measure . We now show that $\{ \Po_y \}_{y\in \Ot}$ defined by
\[
\Po_y = \frac{1}{2}\left( \delta_{x_y^+} + \delta_{x_y^-} \right).
\]
is a continuous measure disintegration of $\Po$.   It is clear that, given $y$, the measure $\Po_y$ is a Borel measure supported on $\Oo_y$. Moreover, one has
\[
\Po_y(f)-\Po_{\tilde{y}}(f)  = \frac{1}{2}\left( f(x_y^+) + f(x_y^-)- f(x_{\tilde{y}}^+) - f(x_{\tilde{y}}^-) \right) 
\]
and since $y \to x^+_{y}$ and $y \to x^-_{y}$ are a continuous maps, for any continuous function
$f$ and any $\epsilon > 0$, one can choose $\delta > 0$, such that $d(y,\tilde{y}) < \delta$ implies $|\Po_y(f)-\Po_{\tilde{y}}(f)| < \epsilon$.

We now show that $\{\Po_y\}$ is indeed a disintegration of $\Po$.  
For $x=(x_i)_{i\ge 0}$, let $\bar{x}=(\bar{x}_i)$ with $\bar{x}_i = - x_i$ for all $i\ge 0$; note that $x_y^{-}=\overline{x_y^+}$.
It is sufficient to  validate consistency of disintegration $\{\Pt_y\}$ for indicators of cylindric sets:
\[
\begin{aligned}
\int_\Ot \int_{\Oo_y} &\mathbbm{1}_{[a_0^{n}]}(x) \Po_y(dx) \Pt(dy)  = 
\frac{1}{2}\int_\Ot \Bigl(\mathbbm{1}_{[a_0^{n}]}(x_y^+)+\mathbbm{1}_{[a_0^{n}]}(x_y^-)\Bigr)  \Pt(dy)\\
& = 
\frac{1}{2}\int_{\Oo} 
\Bigl(\mathbbm{1}_{[a_0^{n}]} (x_{\pi(\tilde{x})}^+)+\mathbbm{1}_{[\bar{a}_0^{n}]} (x_{\pi(\tilde{x})}^+ ) 
\Bigl) \Po(d\tilde{x}) 
 = \frac{1}{2}\int_\Oo \mathbbm{1}_{[a_0^{n}] \cup [\bar{a}_0^{n}]}\left(\tilde{x}\right)  \Po(d\tilde{x})  = \int_\Oo \mathbbm{1}_{[a_0^{n}]}(\tilde{x}) \Po(d\tilde{x}).
\end{aligned}
\] 
Hence the $\Po$  admits for a continuous  disintegration. This example only works for a very specific parameter value $p=1/2$. Interestingly, there exists another example that has exactly the same continuous measure disintegration. Let $p \in (0,1)$ and $\left\{ X_n \right\}_{n\in \Zp}$ be a Markov chain taking values in $\{-1,1\}$, with the transition probability matrix 
\[
P = \left( \begin{array}{cc} p & 1-p \\ 1-p & p \end{array} \right).
\]
The stationary distribution is the distribution $\rho = \left(\frac{1}{2}, \frac{1}{2}\right)$. Then the factor process 
\[
Y_n = \pi\left(X_n, X_{n+1}\right) = X_n \cdot X_{n+1}.
\]
is Bernoulli for all values of $p \in (0,1)$. Let $M_+(w_n^m) = \sum\limits_{i=n}^{m}\mathbbm{1}_{+1}(w_i)$ and $M_-(w_n^m) = \sum\limits_{i=n}^{m}\mathbbm{1}_{-1}(w_i)$, then
\[
\begin{aligned}
\Pt(Y_0 = y_0|Y_1^n = y_1^n) & = \frac{\Pt(Y_0^n = y_0^n)}{\sum\limits_{w \in \{-1,1\}}\Pt(Y_0^n = w y_1^n)} \\
& = \frac{\Po(X_0^{n+1} = (x_y^+){}_0^{n+1})+\Po(X_0^{n+1} = (x_y^-){}_0^{n+1})}{\sum\limits_{w \in \{-1,1\}} \Po(X_0^{n+1} = (x_{w y_1^\infty}^+){}_0^{n+1})+\Po(X_0^{n+1} = (x_{w y_1^\infty}^-){}_0^{n+1})} \\
& = \frac{2 p^{M_+(y_0^n)}(1-p)^{M_-(y_0^n)}}{2 p^{M_+(y_0^n)}(1-p)^{M_-(y_0^n)}+2 p^{M_+(\bar{y}_0y_1^n)}(1-p)^{M_-(\bar{y}_0y_1^n)}} \\
& = \left\{ \begin{array}{cc} p &: y_0 = +1 \\ 1-p &: y_0 = -1, \end{array}\right.
\end{aligned}
\]
for any $n \geq 1$, where we again used the notation $\bar{y}_0 = -y_0$. It follows that the process $\{ Y_n \}_{n \in \Zp}$ is Bernoulli  with the parameter $p$. 
Note that this example has exactly the same fibre structure as the last example: $\Oo_y = \pi^{-1}(y)=\left\{ x_y^+, x_y^- \right\}$, where
\[
x_y^+ = (1,y_0,y_0\cdot y_1,...), \quad x_y^- = (-1,-y_0,- y_0\cdot y_1,...).
\]
Moreover, the same continuous measure disintegration exists: $\{ \Po_y \}_{y \in \Ot}$ with
\[
\Po_y = \frac{1}{2}\left( \delta_{x_y^+} + \delta_{x_y^-} \right).
\]
Continuity and consistency follow by an identical computation as for the Furstenberg example above. Therefore, we have another example of a factor measure with a continuous measure disintegration, but 
without fibre mixing conditions. 

\subsection{Markov factor without continuous measure disintegration} We now show by an example that existence of a continuous measure disintegration is not  necessary. In this example, the factor measure $\Pt$ is Markov. 
Let $\{ X_n \}_{n \in \Zp}$ be a stationary Markov chain taking values in $\Ao=\{ 1,2,3,4 \}$ defined by the probability transition matrix:
\[
P = \left( \begin{array}{llll}
\frac{1}{2} & \frac{1}{2} & 0 & 0\\
\frac{1}{2} & 0 & \frac{1}{2} & 0 \\
0 & 0 & \frac{1}{2} & \frac{1}{2} \\
\frac{1}{2} & 0 & \frac{1}{2} & 0
\end{array} \right).
\]
Define the factor map $\pi$ as follows: let $\At = \{ a,b,c \}$ and put  $\pi:\Ao \to \At$, by $\pi(1) = \pi(3) = a$, $\pi(2) = b$, $\pi(4) = c$. Then the space $\Ot \subset \At^\Zp$ is a subshift of finite type with forbidden words $\{ bb,cc,bc,cb \}$.

This example is not lumpable as $(1,0,0,0)$ is an initial distribution for which the factor process is not Markov; the transition from state $a$ to state $c$ in the output process has probability $0$ until the first occurrence of the word $ba$. However, direct application of the result in \cite{Kemeny60}, shows that the stationary chain $\{X_n\}$ is weakly lumpable with respect $\pi$, i.e., $\{Y_n=\pi(X_n)\}$ is a Markov process, and one can easily compute the corresponding transition probability matrix $\tilde P$.  The stationary invariant  distribution of $\{X_n\}$ is $ p = \left(\frac{1}{3},\frac{1}{6},\frac{1}{3},\frac{1}{6} \right)$. 
Hence, $\Pt$ is a Markov measure with the probability transition matrix 
\[
\tilde P = \left(\begin{array}{lll}
\frac{1}{2} & \frac{1}{4} & \frac{1}{4} \\
1 & 0 & 0 \\
1 & 0 & 0 \\
\end{array} \right).
\]

We proceed by showing that no continuous measure disintegration exists. In this particular case,
the map $\pi$  is \emph{ finite to one factor map}, meaning the fibres have a bounded number of elements. 

Since $\pi^{-1}(b) = 2$ and $\pi^{-1}(c) = 4$, but $\pi^{-1}(a) = \{ 1,3 \}$. If we assume that $y_n = a$ and $m = \min\{ m > n: y_m \in \{b,c \} \}$  is finite, then $\pi^{-1}(y)_n = 1$  if $y_m = b$,
 and $\pi^{-1}(y)_n = 3$ if $y_m = c$. Therefore elements $y \in \Ot$ are uniquely decodable if $y \neq y_0^{n-1}a_n^\infty$ for any $n \ge 1$, i.e., if $y$ does not end with infinite string of $a$'s.
 Otherwise, if $y=y_0^{n-1}a_n^\infty$ for some $n\ge 1$, then the fibre contains exactly
 two points, corresponding to one of the two possible tails: infinite number of $1$'s or $3$'s. 
 This fibre structure makes a continuous measure disintegration of $\Po$ impossible: 
Suppose $\{ \Po_y \}_\Zp$ is a measure disintegration of $\Po$ for the  factor map $\pi$. Then $\Po_y$ is supported on $\Oo_y$ for each $y \in \Ot$. Furthermore, consider the point $z= a_0^\infty$. Then any open neighborhood of $z$ contains, for some $n>0$ the cylinder sets $[a_0^n b]$ and $[a_0^n c]$. For each $y \in [a_0^n b]$ we have $\Oo_{y} \subset [1_0]$, while for each $y'\in [a_0^n c]$ we have $\Oo_{y'} \subset [3_0]$. Hence
\[
\left|\int_{\Oo_{y}} \mathbbm{1}_{[1_0]}(x) \Po_{y}(dx) - \int_{\Oo_{y'}} \mathbbm{1}_{[1_0]}(x) \Po_{y'}(dx)\right| =|1-0|= 1.
\]
Since both cylinders $[a_0^n b]$ and $[a_0^n c]$ have positive $\Pt$-measure, and $\mathbbm{1}_{[1_0]}$ is a continuous function we conclude that no measure disintegration $\Po$ can be coninuous at $z=a_0^\infty$.

One can also use entropy methods to conclude that the factor measure is well behaved. Suppose a subshift of finite type is defined by a primitive $0/1$-matrix $M$. Then the \emph{ topological entropy }of the SFT $X_M$ is equal to $\log(\lambda)$, where $\lambda$ is the largest eigenvalue of $M$. In our case,
\[
M = 
\left( \begin{array}{llll}
1 & 1 & 0 & 0\\
1 & 0 & 1 & 0 \\
0 & 0 & 1 & 1 \\
1 & 0 & 1 & 0
\end{array} \right)
\]
has characteristic polynomial $\lambda^4-2 \lambda^3 -\lambda^2+2 \lambda$, with the largest eigenvalue equal to $2$. The topological entropy is an upper bound for the entropy of measures on the SFT, $h_\Po \leq h_{top}(\Oo) = \log(2)$. For Markov measures we have
\[
h_\Po = -\sum_{x_0,x_1 \in \Ao} p_{x_0} P_{x_0,x_1} \log(P_{x_0,x_1}),
\]
which in our case is  $h_\Po=\log(2)$. For an irreducible SFT the measure of maximal entropy, also known as Parry measure, is unique and is Markov. Moreover, a finite-to-one factor map between two SFTs sends the measure of maximal entropy to the measure of maximal entropy. Thus since $\Pt$ is a measure of maximal
entropy on $\Oo$, then so is  $\nu=\mu\circ\pi^{-1}$, and hence $\nu$ is also Markov.
\section{Conclusions and final remarks}
In the present paper we have established sufficient conditions (Corollary 3.3) for 1-block factors
of Markov measures to be $g$-measures.  The result combines
naturally two types of sufficient conditions: namely, 
the lumpability and  the existence of a
continuous disintegrations. 
We have presented an example showing that these codntions
are in fact complementary. We have also demonstrated that the known sufficient conditions on
regularity of factors of Markov measures imply existence of a continuous measure disintegrations. Note also that Theorem \ref{thm:decomp} does in fact provide 
the necessary and sufficient conditions for the factor measure $\nu$ to be regular. Namely, the factor measure $\nu$ is a $g$-measure if and only if there exists a continuous normalized function  $g:\Ot\to (0,1)$ such that
$g(y)=\tilde g(y)$ for $\Pt$-a.a. $y\in\Ot$, where $\tilde g(y)$ is given by (\ref{tildeg}). Equivalently, we can conclude that $\Pt$ is regular if and only if 
there exists a  disintegration $\{\Po_y\}$ of $\Po$ such the right hand side in (\ref{tildeg}) defines a continuous function of $y$.
It would be interesting to understand whether the two sets of sufficient
conditions for continuity of $\tilde g$ identified in the present paper, are complete, i.e., exhaust all possibilities.

An important point which we have not addressed is the following: in case the one-block factor of the Markov measure is a $g$-measure,
how `smooth' is the corresponding $g$-function?
In all known examples, factors of Markov measures
are either $g$-measures
 for some H\"older continuous $g$, or not regular (c.f., Furstenberg's example).
It would be interesting to understand whether this apparent dichotomy can be turned into a rigorous result.
We believe that the  proposed method: the study of conditional measure disintegrations, can be used to address such questions as well. For example, using the method of Fan-Pollicott, discussed in Section 4.2, one can show that
in the fibre-mixing case, the family of non-homogeneous equilibrium states $\{\mu^y\}$ is in fact `smooth',
and the resulting $g$-function is H\"older continuous. For details on how to deduce properties of the $g$-function
from the properties of the corresponding measure disintegration the reader can consult \cite{Verbitskiy11}.
We have chosen not to provide the details here because the result is well-known and the estimates on the exponential  rate of decay of variations the $g$-function
are weaker in comparison to the more direct proofs \cite{Yoo10},  see also
\cite{Kempton11,Pi1, Pi2} for more general results. 

An advantage of the proposed method lies in the fact that uniqueness
of the Markov measure on each fibre immediately implies existence of a continuous
measure disintegration, and hence, regularity of the factor measure.
In comparison, some earlier results had to establish uniqueness of fibre measures
first, followed by a separate argument for continuity of family of conditional measures, e.g., \cite{Verbitskiy11}.

{
The approach developed in the present paper can also be applied to the study 
of regularity properties of the renormalized Gibbs random fields \cite{BerghoutThesis,Berghout22}.
The so-called \emph{non-overlapping block renormalization transformations} \cite{Enter93} can be represented 
as 1-block factors, i.e., the renormalized field is given by $Y_{n}=\pi(X_n)$   for all $n\in\Z^d$. Existence of a continuous measure disintegration remains a sufficient
condition for regularity of the factor measure, and is implied by the uniqueness of Gibbs measures on fibres. As a result, one obtains significant simplification of the proofs of regularity of Gibbs factors in several cases, e.g., decimations of the Ising models \cite{Enter93,HK}  and  Fuzzy Potts model \cite{Maes95,Ha}.
}
\section*{acknowledgements}
The authors acknowledge support from The Dutch Research Council (NWO), grant 613.001.218.

\end{document}